\documentclass[12pt,reqno]{article}
\usepackage[pdftex]{hyperref}
\usepackage{amsmath, amsthm, graphicx,amsfonts, amssymb,color}
\usepackage{bm}

\newtheorem{thm}{Theorem}[section]

\newtheorem{lem}[thm]{Lemma}
\newtheorem{prop}[thm]{Proposition}
\theoremstyle{definition}

\theoremstyle{remark}
\newtheorem{rem}[thm]{Remark}
\theoremstyle{example}

\numberwithin{equation}{section}
\usepackage{mathrsfs}
\newcommand{\scr}[1]{\mathscr #1}

\newcommand{\E}{\mathbb{E}}

\def\F{\scr F}

\def\e{\scr E}

\renewcommand{\P}{\mathbb P}

\def\bg{\begin}
\def\be{\bg{equation}}
\def\de{\end{equation}}

\def\edar{\end{eqnarray}}

\def\lb{\label}
\def\ct{\cite}
\def\l{\left}\def\r{\right}
\def\fr{\frac}
\def\alp{\alpha}
\def\bt{\beta}

\def\Gm{\Gamma}

\def\lmd{\lambda}
\def\Lmd{\Lambda}

\def\sgm{\sigma}

\def\vph{\varphi}

\def\Omg{\Omega}
\def\fa{\forall}

\def\ift{\infty}

\def\rar{\rightarrow}

\def\q{\quad}

\def\diag{\text {\rm diag}}
\def\var{\text {\rm var}}

\def\lan{\langle}\def\ran{\rangle}

\def\[{\l[} \def\]{\r]}
\def\({\l(} \def\){\r)}
\def\|{\bigg|}

\def\hat{\widehat}
\def\bar{\overline}
\def\tld{\widetilde}

 \def\beqlb{\begin{eqnarray}}\def\eeqlb{\end{eqnarray}}
 \def\beqnn{\begin{eqnarray*}}\def\eeqnn{\end{eqnarray*}}

\title{\bf  {Variational Formulas of Asymptotic Variance for General Discrete-time Markov Chains}}

\author{
{\bf Lu-Jing Huang}\\
{\small College of Mathematics and Informatics, Fujian Normal University }\\
{\bf Yong-Hua Mao\footnote{Corresponding author: maoyh@bnu.edu.cn}}\\
\footnotesize{Laboratory of Mathematics and Complex Systems,
 Ministry of Education,}\\
{\small School of Mathematical Sciences, Beijing Normal University}
}

\date{}

\begin{document}

 \maketitle


\bg{abstract}

The asymptotic variance is an important criterion to evaluate the performance of Markov chains, especially for the central limit theorems. We give the variational formulas for the asymptotic variance of discrete-time (non-reversible) Markov chains on general state space.  The variational formulas provide many applications,  extending the classical Peskun's comparison theorem to non-reversible Markov chains, and obtaining several comparison theorems between Markov chains with various perturbations.

\end{abstract}




\section{Introduction}\lb{introd}

Let $X=(X_n)_{n\geq0}$ be a discrete-time irreducible ergodic Markov chain on a probability space $(\Omg,\F,\P)$, taking values in a measurable state space $(E,\e)$.
Denote $P(x,A),\ x\in E, A\in\e$ and $\pi$ by the associated regular probability transition kernel and stationary distribution respectively.
Let $L^2(\pi)$ be the space of square integrable functions with scalar product $\lan f,g\ran:=\int_Ef(x)g(x)\pi(dx)$, and $L^2_0(\pi)$ be the subspace of $L^2(\pi)$ with mean-zero functions, i.e.,
$$
L^2_0(\pi)=\Big\{f:\ \pi(f):=\int_Ef(x)\pi(dx)=0, \pi(f^2)<\ift \Big\}.
$$
Now $P$ can be viewed as a Markov operator on $L^2(\pi)$ defined by
$$
Pf(\cdot)=\int_E P(\cdot,dy)f(y),\quad f\in L^2(\pi).
$$
Let $P^*$ be the adjoint operator of $P$ in $L^2(\pi)$, which is the probability transition kernel for the time-reversal chain $X^*=(X^*_n)_{n\geq0}$ of $X$, such that
$$
\P(X_0\in A_0,X_1\in A_1,\cdots, X_n\in A_n)=\P(X^*_0\in A_n,X^*_1\in A_{n-1},\cdots, X^*_n\in A_0)\quad
$$
for $n\geq0$ and $A_i\in\e,\ i=0,1,\cdots,n$. The chain $X$ is called reversible, if $P=P^*$, or, Markov operator $P$ is self-adjoint on $L^2(\pi)$.

In this paper, we are interested in the asymptotic variance of chain $X$ for a given $f\in L^2_0(\pi)$:
$$
\widetilde{\sgm}^2(P,f):=\lim_{n\rightarrow\infty}\frac{1}{n}\E\l[\l(\sum_{k=0}^{n-1}f(X_k)\r)^2\r].
$$
Note that the above limit always exists, may be infinite. In the literature, the asymptotic variance plays an important role in functional central limit theorem(FCLT).
Actually, if the FCLT for chain $X$ and $f$ holds; i.e., $n^{-1/2}\sum_{k=0}^{n-1}f(X_k)$ converges weakly to a Gaussian distribution, then $\widetilde{\sgm}^2(P,f)<\infty$ is the variance of this Gaussian distribution, see e.g. \cite{Co72,IL71,RR97}. See also \ct[Chapter 17]{MT09}. Especially for reversible Markov chains, the finiteness of the asymptotic variance can derive the FCLT directly, which is known as Kipnis-Varadhan's condition. Cf. \cite{KV86}. The asymptotic variance also has been widely used in Markov chain Monte Carlo, since it is a popular criterion to evaluate the convergence efficiency of the law $X_n$ to $\pi$ as $n$ goes to infinity, see e.g. \cite{CH13,FHY92,H06,SGS10}.

It is well known that the asymptotic variance has closed relation with Poisson's equation
\begin{equation}\label{poi-int}
(I-P)\phi=f.
\end{equation}
Indeed, if \eqref{poi-int} has a solution $\phi$ in $L^2(\pi)$, then the FCLT holds for $X$ and $f$. Furthermore,
\begin{equation}\label{av-int}
\widetilde{\sgm}^2(P,f)=2\lan \phi,f\ran-\lan f,f\ran.
\end{equation}
See e.g. \cite{Go69,GL78,KLO12} for more details. So from \eqref{av-int}, the asymptotic variance for reversible Makrov chains can be presented by a spectral calculation, which is really helpful in many applications, see e.g. \cite{DL01,Pe73,RR08}. It is not surprising that the study of the asymptotic variance for non-reversible case is more challenging since the spectral theory is lacked. Motivate by that, we will give some variational formulas of the asymptotic variance for non-reversible Markov chains, by solving Poisson's equation \eqref{poi-int} in this paper.

The history of variational formulas, based on Dirichlet's or Poisson's equations, goes back at least to Griffeath and Liggett \cite{GL82}. It presents a Dirichlet principle for capacity. In \cite[Chapter 3]{AF02}, there are also some variational formulas for hitting times. Note that those results cited above are restricted to reversible Markov chains. In non-reversible case, \cite{Do94} studied the variational formulas for some Dirichlet's equations, and  \cite{GL14} provides another type of variational principle of the capacity.
In \cite{HKM20+,HM18}, we considered some Poisson's equations and gave some variational formulas of hitting times for non-reversible Markov processes.  From those works, one can find that the main idea, to establish some variational formulas of Dirichlet's or Poisson's equations for non-reversible Markov chains, heavily depends on its dual Markov chains. Exactly, we need start from a pair of Poisson's equations of chain and its dual chain together(see \eqref{poi2} below).

To ensure the existence of the solutions of associated Poisson's equations, we assume the spectral radium of $P$ in $ L_0^2(\pi)$, $r(P)<1$.
Since $\lan f,f\ran$ is a constant for fixed $f\in L^2_0(\pi)$, we only need to focus on quantity
$$
\sigma^2(P,f):=[\widetilde{\sgm}^2(P,f)+\lan f,f\ran]/2.
$$
Define the bilinear form  $D$ on $L^2(\pi)\times L^2(\pi)$ by
$$
D(\xi,\eta)=\lan (I-P)\xi,\eta\ran,\quad \xi,\eta\in L^2(\pi).
$$

By assuming $r(P)<1$, we obtain our first main result, which gives a variational formula of $\sigma^2(P,f)$ for general Markov chains.

\begin{thm}\lb{main-vf}
Let $X$ be an irreducible ergodic Markov chain on $E$, with probability transition kernel $P$ and stationary distribution $\pi$. Assume that $r(P)<1$. Then
\be\lb{vf-g}
1/\sgm^2(P,\, f)=\inf_{\xi\in \mathcal{M}_{f,1}}\sup_{\eta\in\mathcal{M}_{f,0}}D(\xi+\eta,\xi-\eta)
\de
for $f\in L^2_0(\pi)$, where for $\delta=0,1,$
\be\lb{mf}
\mathcal{M}_{f,\delta}:=\{\xi\in L^2(\pi):\pi(f\xi)=\delta\}.
\de
In particular, if $X$ is a reversible Markov chain, then \eqref{vf-g} reduces to following simple form:
\begin{equation}\label{vf-gr}
1/\sgm^2(P,f)=\inf_{\xi\in \mathcal{M}_{f,1}}D(\xi,\xi).
\end{equation}
\end{thm}

\begin{rem}\lb{rempp*}
\bg{itemize}

\item[(1)] As we mentioned above, the fundamental ingredient for the proof of Theorem \ref{main-vf} is the well posedness of Poisson's equations \eqref{poi2} below.  We use some ideas from \cite{Do94}, which dealt with the energy of Markov chains.

\item[(2)]  Roughly speaking, the irreversibility leads to the ``inf-sup'' form, while only ``inf'' in the reversible case. This phenomenon also happens to principal eigenvalue, capacity and hitting times, see e.g. \cite{Do94,GL14,HM18,KM99,Pi88}. Moreover, since
    $$
    D(\xi+\eta,\xi-\eta)=\lan \xi-(-\eta),(I-P)^*(\xi+(-\eta))\ran,
    $$
    for chain $X$ and $X^*$ we have
   $$
    \widetilde{\sigma}^2(P,f)=\widetilde{\sigma}^2(P^*,f),\quad f\in L^2_0(\pi).
   $$
\end{itemize}
\end{rem}

Variational principle for the asymptotic variance has been researched in Komorowski, Landim and Olla \cite[Chapter 4]{KLO12}. The authors' main idea is from a variational formula for bounded positive definite operators in analysis. Here we make the best of the relation between the Poisson's equations and the asymptotic variance, and obtain a new variational formula. In  Theorem \ref{main-vf}, the strong condition $r(P)<1$  ensures the existence of the solutions of \eqref{poi2}. However for the reversible Markov chains, \cite[Chapter 4]{KLO12} replaces it by the limit of the solutions to its resolvent under some regular conditions.
For that, in second main result we consider reversible Markov chains, and give another variational formula of the asymptotic variance getting rid of Poisson's equation \eqref{poi2}.
The main idea is to construct a continuous-time Markov chain corresponding to the generator $P-I$, so that $\sigma^2(P,f)$ can be represented by an integral of the associated semigroup.

\begin{thm}\lb{main-rev}
Let $X$ be an irreducible ergodic reversible Markov chain on $E$, with probability transition kernel $P$ and stationary distribution $\pi$. Then for $f\in L^2_0(\pi)$,
\begin{equation}\label{vf-rev}
1/\sigma^2(P,f)=\inf_{\xi\in\mathcal{M}_{f,1}}D(\xi,\xi),
\end{equation}
where $\mathcal{M}_{f,1}$ is given by (\ref{mf}).
\end{thm}

 We note that the finity of the asymptotic variance is not required in Theorem \ref{main-rev}, that means \eqref{vf-rev} still holds while $\sigma^2(P,f)=\infty$.

A direct application of our variational formulas is  to compare the asymptotic  efficiency of reversible and non-reversible Markov chains on general state space. We will show that adding some anti-symmetric perturbations to a reversible Markov chain reduces the asymptotic variance, that is, the reversible Markov chain can be accelerated by adding anti-symmetric perturbations, see Theorems \ref{nonre} and \ref{nonre-para} below. In fact, the similar results were stated in e.g. \cite{Bi16,CH13,H06,SGS10} for Markov chains on finite state space, and e.g. \cite{DLP16,HNW15,PH13} for diffusions. As we see, most of those results are based on spectral theory. Here the variational formulas help us keep away from it. On the other hand, we will provide a new method to accelerate reversible Markov chains from Peskun's theorem{ (Theorem \ref{pesk})}, see Corollary \ref{another} below. The main idea is also adding some perturbations to the chain, however these perturbations are complete different with the anti-symmetric perturbations in Theorem \ref{nonre}.

To illustrate another interest of the above variational formulas for the asymptotic variance,   we introduce a partial order, which is known as Peskun's order. Let $P_1$ and $P_2$ are probability transition kernels with same stationary distribution $\pi$, we say that $P_2$ dominates $P_1$ on the off-diagonal, denoted by $P_1\preceq P_2$, if
$$
P_1(x,A\setminus\{x\})\leq P_2(x,A\setminus\{x\})
$$
for all $A\in\e$ and $\pi-$a.e. $x\in E$. This ordering allows some efficiencies of different transition kernels to be compared. For example, we refer reader to \cite{Do94} for capacity, and \cite{HM18} for hitting times. In following, we pay attention to the asymptotic variance, which has been studied for a long time. Indeed, Peskun \cite{Pe73} obtained that $P_2$ is uniformly better than $P_1$ when state space $E$ is finite, while Tierney \cite{Ti98} extended it to general state space. Recently, there are also some researches in this field, see \cite{AL19,LM08,MDO14} etc.. But notice that most of those results are restricted in reversible case. { Using Theorem \ref{main-vf}}, we generalize the comparison result to non-reversible case.

\bg{thm}\lb{pesk}
\textbf{(Peskun's Theorem)}. Let $P_1$ and $P_2$ be probability transition kernels on $(E,\e)$ sharing same stationary distribution $\pi$. Assume that $P_1$ is reversible, $\lambda(P_2)<1$ and $P_1 \preceq P_2$. Then
$$
\widetilde{\sgm}^2(P_2,f)\leq \widetilde{\sgm}^2(P_1,f),\quad \text{for all }f\in L^2_0(\pi).
$$
\end{thm}

\begin{rem}\label{pesk-non}
It is natural to ask if the Peskun's theorem holds for the cases that $P_1$ is non-reversible and $P_2$ is reversible, or $P_1$ and $P_2$ are both non-reversible? Unfortunately, the answer is not. We will give two counter examples(see Example \ref{chen} and \ref{two-nonrev} below) to illustrate that.
\end{rem}

The remaining part of this paper is organized as follows.
In Sect.\ref{av_pe} we present some sufficient conditions for the existence of the solution of \eqref{poi-int}, thus the finiteness of the asymptotic variance.
In Sect.\ref{main1} we give a proof of Theorem \ref{main-vf}, meanwhile we give another type of variational formula of the asymptotic variance.
Sect.\ref{dt_2} is devoted to the proof of Theorem \ref{main-rev}.
In Sect.\ref{appli}, we
obtain comparison theorems between reversible and non-reversible Markov chains, give a proof of Theorem \ref{pesk} and some another applications of our main results,  and present a new method of accelerating reversible Markov chains. In the final section, we give a conclusion and some examples to compare the methods of accelerating reversible Markov chains.

\section{The asymptotic variance and Poisson's equation}\lb{av_pe}

Recall that $X=(X_n)_{n\geq0}$ is a discrete-time irreducible ergodic Markov chain on a probability space $(\Omega,\F,\P)$, taking values on state space $(E,\e)$. Fix a function $f\in L^2_0(\pi)$, we consider Poisson's equation \eqref{poi-int}.
As we mentioned before, if \eqref{poi-int} has a solution in $L^2(\pi)$, then the asymptotic variance must be finite and has representation \eqref{av-int}. So we want to make some explicit sufficient conditions of the well-posedness of \eqref{poi-int} in this section. For this, we introduce some notations.

We say that Markov chain $X$ is strongly ergodic if there exist $\rho<1$ and $M<\infty$ such that
$$
||P^n-\pi||_{\ift\rar\ift}:=\sup_{x\in E}||P^n(x,\cdot)-\pi(\cdot)||_{\var}\leq M\rho^n,\q n\geq 1,
$$
where $||\mu(\cdot)||_{\var}=\frac{1}{2}\sup_{|f|\leq 1}|\mu(f)|$. A weaker condition than strong ergodicity is geometric ergodicity. Markov chain $X$ is called geometrically ergodic if
there exist  $\rho<1$ and $M(x)<\infty$ for $\pi-$a.e. $x\in E$ such that
$$
||P^n(x,\cdot)-\pi(\cdot)||_{\var}\leq M(x)\rho^n,\q n\geq 1
$$
 More equivalences of strong ergodicity and geometric ergodicity can be found in \cite{MT09,RR97}.

For brevity, we also denote $||\cdot||_2$ by the $L^2-$norm with respect to $\pi$. Let $r(P)$ be the spectral radium of $P$ in $L_0^2(\pi)$. It is obvious that if $r(P)<1$, then Poisson's equation \eqref{poi-int}  has a solution in $L_0^2(\pi)$.

The following result tells us that some ergodic conditions can ensure $r(P)<1$, thus the existence of the solution of \eqref{poi-int}.
\bg{prop}\lb{uni}
If Markov chain $X$ is geometrically ergodic and reversible or strongly ergodic, then $r(P)<1$. Therefore, Poisson's equation \eqref{poi-int} has a unique solution in $L^2_0(\pi)$ for any $f\in L^2_0(\pi)$.
\end{prop}

\begin{proof}
It is shown in \cite{RR97} that reversible Markov chains are geometrically ergodic if and only if
$$
||P||_{2\rar 2}:= \sup_{f\in L_0^{2}(\pi),||f||_2\leq 1}||Pf||_2<1.
$$
Therefore, if $X$ is reversible and geometrically ergodic, we have $r(P)<1$. That is, $(I-P)^{-1}$ is a bounded one-to-one linear operator from $L^2_0(\pi)$ into $L^2_0(\pi)$. Thus $(I-P)^{-1}f=\sum_{k=0}^\infty P^kf$ is the unique solution to \eqref{poi-int} in $L_0^2(\pi)$.

To prove the second part, we apply  Riesz-Thorin's interpolation theorem.
Let
$$
L_0^{\infty}(\pi)=\{f\in L^\infty(\pi):\pi(f)=0\}
$$
and
$$
L^{1}_0(\pi)=\{f:\pi(f)=0\text{ and } ||f||_1:=\pi(|f|)<\infty\}.
$$
Denote for $ p=1,\ift$
$$
||H||_{p\rar p}= \sup_{f\in L_0^{p}(\pi),||f||_p\leq 1}||Hf||_p.
$$
Now for $X$ being strongly ergodic, there exist constants $M<\infty$ and $\rho\in(0,1)$ such that
$$
||P^n||_{\infty\rar\infty}\leq M\rho^n.
$$
Noting that $||P^n||_{1\rar 1}\leq 1$ by Markov property, Riesz-Thorin's interpolation theorem implies
$$
||P^n||_{2\rar2}\leq ||P^n||^{1/2}_{1\rar1}||P^n||^{1/2}_{\infty\rar\ift}\leq \sqrt{M}(\sqrt{\rho})^n.
$$
Combining this with the definition of $r(P)$ it follows that
$$
r(P)=\lim_{n\rightarrow \infty}||P^n||_{2\rar 2}^{1/n}\leq \sqrt{\rho}<1.
$$
So $(I-P)^{-1}$ can be viewed as a bounded one to one linear operator from $L^2_0(\pi)$ into $L^2_0(\pi)$. That is, for any $f\in L^2_0(\pi)$, $(I-P)^{-1}f$ is the unique solution to \eqref{poi-int} on $L^2_0(\pi)$.
\end{proof}

\begin{rem}
\begin{itemize}
\item[(1)] The second part of Proposition \ref{uni} was shown in \cite[Theorem 3.5]{IL78} while the state space is finite. Here we extend it to general state space by a different method.




\item[(2)] In Proposition \ref{uni}, we provide some sufficient conditions for the existence of the solution of \eqref{poi-int}, thereby the FCLT holds and the corresponding asymptotic variance is finite. We note that here ``$r(P)<1$'' is not too strong for non-reversible Markov chains, since \cite{Ha05} finds a geometrically ergodic Markov chain such that the FCLT does not hold for some $f\in L^2_0(\pi)$.
\end{itemize}
\end{rem}

\section{Variational formulas}

In this section, we give proofs of Theorems \ref{main-vf} and \ref{main-rev} in Sect.\ref{main1} and \ref{dt_2} respectively. We also provide another type of the variational formula of the asymptotic variance in Sect.\ref{main1}. { To simplify the notation we will use $||\cdot||$ in the place of $||\cdot||_2$  in the rest of this paper.}

\subsection{General Markov chains}\lb{main1}

Recall that
$$
\sigma^2(P,f)=[\widetilde{\sgm}^2(P,f)+\lan f,f\ran]/2
$$
for $f\in L^2_0(\pi)$, and $D(\xi,\eta)=\lan(I-P)\xi,\eta\ran$, $\xi,\eta\in L^2(\pi)$ is a bilinear form on $L^2(\pi)\times L^2(\pi)$.

{\bf Proof of Theorem \ref{main-vf}.}
Fix a function $f\in L_0^2(\pi)$. Since $r(P^*)=r(P)<1$, the pair of Poisson's equations
\be\lb{poi2}
\begin{cases}
(I-P)\phi=f;\\
(I-P^*)\phi^*=f
\end{cases}
\de
exist unique solutions $\phi, \phi^*\in L_0^2(\pi)$ respectively. Therefore, we have
\begin{equation}\label{inn-sol}
D(\phi,\phi^*)=\lan f,\phi^*\ran=\lan\phi,f\ran=\sigma^2(P,f).
\end{equation}
To simplify the notation, we set
$$
\bar\phi=\frac{1}{2\lan\phi,f \ran}(\phi+\phi^*),\q \hat{\phi}=\frac{1}{2\lan\phi,f \ran}(\phi-\phi^*).
$$
It is easy to verify that $\bar\phi\in\mathcal{M}_{f,1}$ and $\hat{\phi}\in\mathcal{M}_{f,0}$. Moreover, we have
$$
\bar\phi+\hat{\phi}=\frac{1}{\lan\phi,f \ran}\phi,\q \bar\phi-\hat{\phi}=\frac{1}{\lan\phi,f \ran}\phi^*.
$$

Now for any $\xi\in\mathcal{M}_{f,1}$, let $\xi_1=\xi-\bar{\phi}$. Then it is clear that $\xi_1\in\mathcal{M}_{f,0}$. Combining this with the definitions of $\phi$ and $\phi^*$, we get
\begin{equation}\label{phi-xi}
D(\phi,\xi_1)=\lan f, \xi_1 \ran=0,\q
D(\xi_1,\phi^*)=\lan \xi_1,(I-P)^*\phi^* \ran=\lan \xi_1,f \ran=0.
\end{equation}

Thanks to \eqref{inn-sol}-\eqref{phi-xi} and the fact $D(\xi_1,\xi_1)\geq0$,
$$
\aligned
D(\xi+\hat\phi,\xi-\hat\phi)
&=D(\xi_1+\bar\phi+\hat\phi,\xi_1+\bar\phi-\hat\phi)\\
&=D(\xi_1,\xi_1)+\frac{1}{|\lan\phi,f\ran|^2}D(\phi,\phi^*)\\
&\geq \frac{1}{|\lan\phi,f\ran|^2}D(\phi,\phi^*)=\frac{1}{\sigma^2(P,f)},
\endaligned
$$
which implies
\be\lb{vfleq}
1/\sgm^2(P,f)\leq\inf_{\xi\in \mathcal{M}_{f,1}}\sup_{\eta\in \mathcal{M}_{f,0}} D(\xi+\eta,\xi-\eta).
\de

Conversely, for any $\eta\in\mathcal{M}_{f,0}$, let $\eta_1=\eta-\hat{\phi}\in\mathcal{M}_{f,0}$. Replacing $\xi_1$ by $\eta_1$ in \eqref{phi-xi}, similar with above analysis we arrive at
$$
D(\bar\phi+\eta,\bar\phi-\eta)=\frac{1}{|\lan\phi,f\ran|^2}D(\phi,\phi^*)-D(\eta_1,\eta_1)\leq\frac{1}{\sigma^2(P,f)},
$$
which implies
\be\lb{vfgeq}
1/\sgm^2(P,f)\geq\inf_{\xi\in \mathcal{M}_{f,1}}\sup_{\eta\in \mathcal{M}_{f,0}} D(\xi+\eta,\xi-\eta).
\de
Combining \eqref{vfleq} and \eqref{vfgeq}, we obtain \eqref{vf-g}.

In particular, if $X$ is reversible, that is, $P$ is self-adjoint in $L^2(\pi)$, then obviously
$$
D(\xi+\eta,\xi-\eta)=D(\xi,\xi)-D(\eta,\eta)\q \text{for any }\xi,\eta\in L^2(\pi).
$$
Thus we can see that for fixed $\xi\in\mathcal{M}_{f,1}$, the supremum in \eqref{vf-g} is attained at $\eta=0$. So we complete the proof.
\quad $\square$

Next, we give another type of variational formula for the asymptotic variance, which may not be derived directly from \eqref{vf-g}. To do that, we need do some preparations. Define the reversibilisation of kernel $P$ by
$$
K(x,dy)=\frac{1}{2}[P(x,dy)+P^*(x,dy)],\q x,y\in E.
$$

Since $K$ is self-adjoint on $L^2_0(\pi)$, $(I-K)^{-1}$ is well-defined through the spectral representation of $K$. Indeed, let
$E_\lmd (-1\leq\lmd<1)$ be the spectral (projection) measure of $K$ on $L^2_0(\pi)$. We have
$I-K=\int_{[-1,1)}(1-\lmd)d E_\lmd$, so
$$
(I-K)^{-1}=\int_{[-1,1)}\fr1{1-\lmd}d E_\lmd
$$
is well-defined.
That is,
for $f\in L^2_0(\pi)$,
$$
\lan f,(I-K)^{-1}f\rangle=\int_{[-1,1)}\fr1{1-\lmd}d \lan E_\lmd f,f\rangle:=\lim_{n\rar\ift}\int_{[-1,1-1/n)}\fr1{1-\lmd}d \lan E_\lmd f,f\rangle.
$$
The above limit exists (may be infinite), since the last integral is finite for each $n$, and is increasing as $n$ goes to infinity.

Now let $T=(I-P)(I-K)^{-1}(I-P)^*$. Since $(I-P)^*$ maps $L^2_0(\pi)$ into itself, we see that $T$, as a composition of $I-P, (I-K)^{-1}$ and $(I-P)^*$, is well-defined in such way:
$\fa f\in L^2_0(\pi)$, we have $g:=(I-P)^*f\in L^2_0(\pi)$ and
$$
\lan f,T f\rangle=\lan (I-P)^*f,(I-K)^{-1}(I-P)^*f\ran =\lan g, (I-K)^{-1}g\ran.
$$

 The previous argument enables one to consider the Poisson's equation concerning $T$. The following lemma is a key, which presents the relation of solutions to Poisson's equation of $T$ and \eqref{poi2}.

\begin{lem}\lb{inv}
Let $f\in L^2_0(\pi)$ such that \eqref{poi2} have solutions $\phi,\ \phi^*$ in $L^2(\pi)$ respectively. Then $\bar{\phi}:=(\phi+\phi^*)/2$ satisfies $T\bar{\phi}=f$.
\end{lem}

\begin{proof}
Since $\phi,\ \phi^*\in L^2(\pi)$ are solutions to Poisson's equations \eqref{poi2} respectively,
$$
\aligned
(I-P)^*\bar{\phi}&=\frac{1}{2}[(I-P)^*\phi+(I-P)^*\phi^*]=\frac{1}{2}[(I-P)^*\phi+f]\\
&=\frac{1}{2}[(I-P)^*\phi+(I-P)\phi]\\
&=(I-K)\phi,
\endaligned
$$
therefore
$
T\bar{\phi}=(I-P)(I-K)^{-1}(I-K)\phi=(I-P)\phi=f.
$
\end{proof}

Using operator $T$, we present another variational formula for the asymptotic variance in following theorem.

\begin{thm}
Let $X$ be an irreducible ergodic Markov chain on $E$, with probability transition kernel $P$ and stationary distribution $\pi$. Assume that $r(P)<1$. Then
\be\lb{e2}
1/\sgm^2(P,f)=\inf_{\pi(f\xi)=1}\lan\xi, T\xi\ran, \q \text{for\ }f\in L^2_0(\pi).
\de
In particular, if $X$ is reversible, then \eqref{e2} also reduces to \eqref{vf-gr}.
\end{thm}

\begin{proof}
Fix $f\in L^2_0(\pi)$ and denote $\vph=\bar{\phi}/\pi(\bar{\phi}f)$, where $\bar{\phi}$ is defined in Lemma \ref{inv}. From Lemma \ref{inv} and  \eqref{inn-sol}, we get
$$
\lan\vph,T\vph\ran=1/\pi(\bar{\phi}f)=1/\sgm^2(P,f).
$$
Now we claim the infimum in the right-hand side of \eqref{e2} is attained at $\vph$, hence \eqref{e2} holds. Indeed, for any $\xi$ with $\pi(\xi f)=1$, let $\widetilde{\xi}=\xi-\vph$. Then it is easy to check that  $\pi(\widetilde{\xi}f)=0$. Using Lemma \ref{inv} again gives that
$$
\lan\widetilde{\xi},T\vph\ran=\frac{1}{\pi(\bar{\phi}f)}\lan\widetilde{\xi}, f\ran=0.
$$
Therefore, combining it with the fact that $T$ is positive definite on $L^2_0(\pi)$, we arrive at
$$
\lan\xi,T\xi\ran=\lan\vph,T\vph\ran+\lan\widetilde{\xi},T\widetilde{\xi}\ran+2\lan\widetilde{\xi},T\vph\ran\geq \lan\vph,T\vph\ran= 1/\sgm^2(P,f).
$$

In particular, if $X$ is reversible, we have $K=P$. That is, $T=I-P$. Thus the proof is completed.
\end{proof}

\begin{rem}
Note that operator $T$ also appears in \cite{GL14}, which gives a Dirichlet principle of the capacity for non-reversible Markov chains. Although we only use an ``inf'' form in \eqref{e2}, the operator $T$ involves both $P$ and $P^*$.
\end{rem}

\subsection{Reversible Markov chains}\label{dt_2}

In this section, we assume that $X$ is reversible, i.e., $P$ is self-adjoint in $L^2(\pi)$. So $P$ admits a spectral decomposition:
\begin{equation}\label{spec-de}
P=\int_{-1}^1\lmd dE_\lmd.
\end{equation}

Denote $L=P-I$. One can find that $L$ is a bounded self-adjoint linear operator on $L^2(\pi)$, and $(\beta-L)$ is invertible and $||(\beta-L)^{-1}||_{2\rar 2}\leq 1/\beta$ for any $\beta>0$. Thus from Hille-Yoshida Theorem it follows that $L$ is a generator of a strongly continuous contraction semigroup, denoted by $(P_t)_{t\geq0}$. Further, $P_t=\exp(tL)$ is also self-adjoint in $L^2(\pi)$.

For $f\in L^2_0(\pi)$,  $\sigma^2(P,f)$ have the following relation(see e.g. \cite{KV86}):
$$
\sigma^2(P,f)=\int_{-1}^1\frac{1}{1-\lmd}d\lan E_\lmd f,f\ran=\int_0^\infty \lan P_sf,f\ran ds=2\int_0^\infty ||P_sf||^2ds.
$$
 Now we are ready to prove Theorem \ref{main-rev}.

{\bf Proof of Theorem \ref{main-rev}.}
(a)\; Fix $f\in L^2_0(\pi)$. Let
$$
\widetilde{\eta}_t=\int_0^tP_sfds, \q t\ge0.
$$
Since $P_t$ is self-adjoint in $L^2_0(\pi)$, we have
\be\lb{f-eta}
\pi(f\tld \eta_t)=\int_0^t\lan f,P_sf\ran ds=\int_0^t||P_{s/2}f||^2ds=2\int_0^{t/2}||P_{r}f||^2dr,
\de
and similarly
\be\lb{ptf-eta}
\lan P_tf,\tld \eta_t\ran=\int_0^t\lan P_tf,P_sf\ran ds=\int_0^t||P_{(t+s)/2}f||^2ds=2\int_{t/2}^t||P_{r}f||^2dr.
\de

(b)\; If $\sigma^2(P,f)<\ift$, then
$$
\lim_{t\rar\ift}\int_0^{t/2}||P_{r}f||^2dr=
\lim_{t\rar\ift}\int_0^{t}||P_{r}f||^2dr=\sigma^2(P,f)/2<\ift,
$$
so that
\be\lb{y1}
\lim_{t\rar\ift}\fr{\int_{t/2}^t||P_{r}f||^2dr}{\int_0^{t/2}||P_{r}f||^2dr}=\lim_{t\rar\ift}\fr{\int_{0}^t||P_{r}f||^2dr}{\int_0^{t/2}||P_{r}f||^2dr}-1=0.
\de

(c)\; Since $L=P-I$ is the (bounded) generator of the semigroup $(P_t)_{t\geq0}$, the Kolmogorov's forward equation gives that
 $-L\widetilde{\eta}_t=f-P_tf$ for $t>0$. Thus combine this fact with \eqref{f-eta} and \eqref{ptf-eta} we have
$$
\aligned
\lan\widetilde{\eta}_t,(-L)\widetilde{\eta}_t\ran&=\lan\widetilde{\eta}_t,f-P_tf\ran
=\lan\widetilde{\eta}_t,f\ran-\lan\widetilde{\eta}_t,P_tf\ran \\
&=2\Big(\int_0^{t/2}||P_rf||^2dr  -  \int_{t/2}^t||P_rf||^2dr \Big).
\endaligned
$$
By letting
$
 \eta_t=\widetilde{\eta}_t/\pi(f\widetilde{\eta}_t)\in\mathcal{M}_{f,1},$ we have
$$
\lan \eta_t,(-L)\eta_t\ran=\frac{1}{2\int_0^{t/2}||P_rf||^2dr}\bigg(1- \frac{\int_{t/2}^t||P_rf||^2dr}{\int_0^{t/2}||P_rf||^2dr}\bigg).
$$
When $\sigma^2(P,f)<\ift$,  then by the monotone convergence theorem and (\ref{y1}), we have
$$
\lim_{t\rightarrow\infty}\lan \eta_t,(-L)\eta_t\ran=\frac{1}{2\int_0^\infty||P_rf||^2dr}=\frac{1}{\sigma^2(P,f)}.
$$
Otherwise, the both sides of the  above equation are null. In any case,
since $\eta_t\in \mathcal{M}_{f,1}$
we obtain
$$
1/\sigma^2(P,f)\geq \inf_{\xi\in\mathcal{M}_{f,1}}\lan\xi,(-L)\xi\ran.
$$

(d)\; To prove the converse inequality:
$$
1/\sigma^2(P,f)\leq \inf_{\xi\in\mathcal{M}_{f,1}}\lan\xi,(-L)\xi\ran,
$$
we may and do assume that $\sigma^2(P,f)<\ift$.

Set $\xi_t=\xi-\eta_t$ for $\xi\in\mathcal{M}_{f,1}$ and  $t>0$. It is obvious that $\pi(f\xi_t)=0$, and
$$
\lan\xi_t,(-L)\eta_t\ran=\frac{1}{\pi(f\widetilde{\eta}_t)}\lan\xi_t,f-P_tf\ran=-\frac{1}{\pi(f\widetilde{\eta}_t)}\lan\xi-\eta_t,P_tf\ran.
$$
Since $\lim_{t\rar\ift}||P_tf||=\pi(f)=0$, we have
\be\lb{f_xi}
|\lan\xi, P_tf\ran|\leq ||\xi||\cdot||P_tf||\rightarrow0,\q \text{as}\q t\rar \infty.
\de
Therefore, combining \eqref{f_xi} with (a) and (b) gives that
$$
\lim_{t\rightarrow\infty}\lan\xi_t,(-L)\eta_t\ran=\lim_{t\rightarrow\infty} \frac{2}{\pi(f\widetilde{\eta}_t)^2}\int_{t/2}^t ||P_rf||^2dr=0.
$$
As $\lan\xi_t,(-L)\xi_t\ran\geq0$ for  $t>0$,  from (c) we get
$$\aligned
\lan\xi,(-L)\xi\ran&=\lim_{t\rightarrow\infty}\big[\lan\xi_t,(-L)\xi_t\ran+\lan\eta_t,(-L)\eta_t\ran+2\lan\xi_t,(-L)\eta_t\ran\big]\\&\geq \lim_{t\rightarrow\infty}\lan\eta_t,(-L)\eta_t\ran=1/\sigma^2(P,f).
\endaligned
$$
This implies
$$
1/\sigma^2(P,f)\leq \inf_{\xi\in\mathcal{M}_{f,1}}\lan\xi,(-L)\xi\ran.
$$
Thus we complete the proof.$\quad \square$



Now we will provide an alternative proof of Theorem \ref{main-rev} in case of the finite asymptotic variance. The idea is to approximate the asymptotic variance by its resolvent, which may be traced to \cite{KLO12}.


\begin{thm}\lb{limits}
Let $X$ be an irreducible ergodic reversible Markov chain on $E$, with probability transition kernel $P$ and stationary distribution $\pi$. For $f\in L^2_0(\pi)$ satisfying $\sigma^2(P,f)<\infty$, we have
$$
1/\sigma^2(P,f)=\inf_{\xi\in\mathcal{M}_{f,1}}D(\xi,\xi).
$$
\end{thm}

\begin{proof}
Assume $f\in L^2_0(\pi)$ with $\sigma^2(P,f)<\infty$.
Let
$$
P_\beta=(\beta+1)I-P,\q \beta>0.
$$
Then there exists a unique solution
 $\phi_{\beta}\in L^2(\pi)$ such that  $P_\beta \phi_\beta=f$. Use the spectral decomposition \eqref{spec-de} of $P$ we have
 \be\lb{qqa}
 \pi(f\phi_\bt)=\int_{-1}^1\fr1{\bt+1-\lmd}d\lan E_\lmd f,f\ran, \q ||\phi_\bt||^2=\int_{-1}^1\fr1{(\bt+1-\lmd)^2}d\lan E_\lmd f,f\ran.
 \de
 Thus
 \be\lb{limitpf}
\lim_{\beta\rightarrow 0} \pi(f\phi_\beta)=\lim_{\beta\rightarrow 0} \langle P_\beta ^{-1}f,f \rangle
=\lim_{\beta\rightarrow 0}\int_{-1}^1 \frac{1}{\beta+1-\lmd}d\lan E_\lmd f,f\ran
=\sigma^2(P,f).
\de
Since
$$
\fr\bt{(\bt+1-\lmd)^2}=\fr\bt{\bt+1-\lmd}\fr1{\bt+1-\lmd}\le\fr1{1-\lmd },
$$
it follows from \eqref{qqa} and the dominated convergence theorem that
$$
\lim_{\beta\rightarrow 0} \bt ||\phi_\bt||^2=\lim_{\beta\rightarrow 0}\int_{-1}^1\fr\bt{(\bt+1-\lmd)^2}d\lan E_\lmd f,f\ran
=\int_{-1}^1\lim_{\beta\rightarrow 0}\fr\bt{(\bt+1-\lmd)^2}d\lan E_\lmd f,f\ran=0.
$$
By letting $\widetilde{\phi}_{\beta}=\phi_{\beta}/\pi(f\phi_{\beta})$, we have $\widetilde{\phi}_{\beta}\in L^2(\pi)$ and $\pi(f\widetilde{\phi}_{\beta})=1$. So from \eqref{limitpf},
\be\lb{limitgeq}
\lim_{\beta\rightarrow 0} \bt ||\tld\phi_\bt||^2=\lim_{\beta\rightarrow 0} \bt ||\phi_\bt||^2/\pi(f\phi_{\beta})^2=0.
\de
Moreover,  we get
$$
\lim_{\beta\rightarrow 0}D(\tld\phi_\beta, \tld\phi_\beta)=\lim_{\beta\rightarrow 0}\big[\langle P_\beta \tld\phi_\beta,\tld\phi_\beta \rangle-\bt||\tld\phi_\bt||^2\big]
=\lim_{\beta\rightarrow 0}1/\pi(f \phi_\beta)
=1/\sigma^2(P,f).
$$
Therefore,
$$
1/\sigma^2(P,f)\geq \inf_{\xi\in\mathcal{M}_{f,1}}D(\xi,\xi).
$$

For the converse inequality, let $\xi_\beta=\xi-\widetilde{\phi}_\beta$ for any $\xi\in\mathcal{M}_{f,1}$.
Then $\pi(f\xi_\beta)=0$, $\beta>0$ and
$$
\aligned
D(\xi,\xi)&=D(\widetilde{\phi}_\beta+\xi_\beta,\widetilde{\phi}_\beta+\xi_\beta)\\
&=D(\widetilde{\phi}_\beta,\widetilde{\phi}_\beta)+D(\xi_\beta,\xi_\beta)+2D(\widetilde{\phi}_\beta,\xi_\beta)\\
&=D(\widetilde{\phi}_\beta,\widetilde{\phi}_\beta)+D(\xi_\beta,\xi_\beta)-2\beta\langle\widetilde{\phi}_\beta,\xi_\beta  \rangle.
\endaligned
$$
Note that by (\ref{limitgeq})
$$
|\beta\langle\widetilde{\phi}_\beta,\xi_\beta  \rangle|=|\beta\langle\widetilde{\phi}_\beta,\xi-  \widetilde{\phi}_\beta\rangle|\leq \beta||\xi||\cdot ||\widetilde{\phi}_\beta||+\beta||\widetilde{\phi}_\beta||^2\rightarrow0,\quad \text{as}\ \beta\rightarrow 0.
$$
Thus
$$
D(\xi,\xi)=\lim_{\beta\rightarrow0}D(\widetilde{\phi}_\beta,\widetilde{\phi}_\beta)+\lim_{\beta\rightarrow0}D(\xi_\beta,\xi_\beta)\geq1/ \sigma^2(P,f).
$$
That is
$$
\inf_{\xi\in\mathcal{M}_{f,1}}D(\xi,\xi)\geq 1/ \sigma^2(P,f).
$$
\end{proof}




\section{Applications}\lb{appli}

In this section, We will address some applications of the variational formulas in Theorems \ref{main-vf} and \ref{main-rev}. Firstly, in Sect.\ref{appli1} we show that adding anti-symmetric perturbations to reversible Markov chains reduces the asymptotic variance. In Sect.\ref{appli2}, we consider the asymptotic variance under the classical Peskun's order, and give the  proof of Theorem \ref{pesk}. Finally, we present a new method to accelerate reversible Markov chains in terms of the asymptotic variance.


\subsection{Comparison theorems between reversible and non-reversible Markov chains}\lb{appli1}

In this subsection we will construct non-reversible Markov chains by adding anti-symmetric perturbations to reversible ones, and then use our variational formulas to compare their asymptotic variance.

Let $K$ be  a probability transition kernel on $E\times \e$,
which is reversible w.r.t. probability measure $\pi$:
$$
\pi(dx)K(x,dy)=\pi(dy)K(y,dx)=:\widetilde{K}(dx,dy),\ \text{for all }x,y\in E.
$$
Let $\Gm$ be a kernel on $E\times \e$ and define
$$
\widetilde{\Gm}(dx,dy)=\pi(dx)\Gm(x,dy).
$$
We assume that kernel $\Gm$ has following properties:
\bg{itemize}
\item[($a$)] $\Gm(x,E)=0$, $x\in E$;

\item[($b$)] $\widetilde{\Gm}(A,B)=-\widetilde{\Gm}(B,A)$, for all $A,B\in \e$;

\item[($c$)] there is a measurable function $h\in\e\times\e$ such that $|h|\le1$ and
    $$
    \widetilde{\Gm}(dx,dy)=h(x,y)\widetilde{K}(dx,dy).
    $$
\end{itemize}

In \cite{Bi16}, such a  measure $\widetilde{\Gm}$ is called the vorticity measure of $(K,\pi)$, which was used it to construct non-reversible Metropolis-Hastings.

\bg{lem}\lb{ptk}
Suppose that $K$ is a reversible probability transition kernel with stationary distribution $\pi$. Let $\Gm$ be a kernel having the properties $(a)-(c)$. Then
$
P=K+\Gm
$
is a probability transition kernel on $(E,\e)$ with stationary distribution $\pi$.
\end{lem}

\bg{proof}
As $|h(x,y)|\leq 1$ for all $(x,y)$ in $E\times E$,
$$
P(x,A)=K(x,A)+\Gm(x,A)=\int_A [1+h(x,y)]K(x,dy)\geq 0
$$
for all $x\in E, A\in \e$. Moreover, the property $(a)$ requires
$$
P(x,E)=K(x,E)+\Gm(x,E)=1,\  x\in E.
$$
So $P$ is a probability transition kernel on $E\times\e$. For every $A\in \e$, by properties $(a)$ and $(b)$,
$$
\int_E P(x,A)\pi(dx)=\int_EK(x,A)\pi(dx)+\widetilde{\Gm}(E,A)
=\pi(A)-\widetilde{\Gm}(A,E)
=\pi(A).
$$
\end{proof}

Now we compare the asymptotic variance of chains $K$ and $P$.

\bg{thm}\lb{nonre}
Assume that $r(P)<1$, then
$
\widetilde{\sgm}^2(P,f)\leq \widetilde{\sgm}^2(K,f)
$
for each $f\in L^2_0(\pi)$.
\end{thm}

\bg{proof}
It is trivial for $f\in L^2_0(\pi)$ with $\widetilde{\sigma}^2(K,f)=\infty$. Now we assume $\widetilde{\sigma}^2(K,f)<\infty$, that is, $\sigma^2(K,f)<\infty$. From Theorems \ref{main-vf} and \ref{main-rev}, we get
$$
\aligned
1/\sgm^2(P,f)
&=\inf_{\xi\in\mathcal{M}_{f,1}}\sup_{\eta\in \mathcal{M}_{f,0}} \lan(I-P)(\xi+\eta),\xi-\eta\ran\\
&\geq \inf_{\xi\in\mathcal{M}_{f,1}}\lan(I-P)\xi,\xi\ran\\
&=\inf_{\xi\in\mathcal{M}_{f,1}}\lan(I-K)\xi,\xi\ran\\
&=1/\sgm^2(K,f),
\endaligned
$$
where we use the fact: by (b),
$$
\lan \Gm\xi,\xi\ran=\int\tld\Gm(dx,dy)\xi(x)\xi(y)=0.
$$
Thus the proof is completed by noting  $\widetilde{\sgm}^2(P,f)=2{\sgm}^2(P,f)-||f||^2$.
\end{proof}

When state space $E$ is  finite, the result similar to Theorem \ref{nonre} was firstly established in Hwang \cite{H06}, and more extensive arguments can be seen in \cite{Bi16,CH13,SGS10}. As we can see, most of those results are based on the spectral theory. Here the variational formulas help us extend the result to general state space and keep away from spectral theory of the non-reversible Markov chains, since the spectral theory is usually subtle for non-reversible operator.

Next we  introduce a parameter to control the anti-symmetric perturbations. For this, let $K$ and $\Gamma$ be defined as above. Define a family of probability transition kernels on $E$ by
\be\lb{para}
P_\alpha=K+\alpha \Gamma, \quad -1\leq\alpha\leq1.
\de

Then all $P_\alp$ share the same stationary distribution $\pi$. As a function of variable $\alpha$, the asymptotic variance of chain $P_\alpha$ has monotone and symmetry properties.

\bg{thm}\label{nonre-para}
Let $P_\alpha$ be defined in \eqref{para} with $\Gamma$ satisfying properties $(a)-(c)$. Assume that $r(P_\alpha)<1$ for all $\alpha\in[-1,1]$. Then
\begin{itemize}
\item[(1)]    $    \widetilde{\sigma}^2(P_\alpha,f)= \widetilde{\sigma}^2(P_{-\alpha},f),$ for  $\alpha\in[-1,1]$ and $f\in L_0^2(\pi)$.

\item[(2)]  Given $f\in L_0^2(\pi)$, $\widetilde{\sigma}^2(P_\alpha,f)$ is non-decreasing for $\alpha\in [-1,0]$.
\end{itemize}
\end{thm}

\begin{proof}
Since $P_{-\alpha}$ is the dual chain of $P_\alpha$, (1) follows  by Remark \ref{rempp*} (2).

For (2), note that
$$
K=(P_\alp+P^*_\alp)/2=(P_\alp+P_{-\alp})/2.
$$
For $\xi,\eta\in L^2(\pi)$,
\be\lb{expa}
\aligned
\lan(I-P_\alpha)(\xi+\eta), \xi-\eta \ran
&=\lan(I-K)(\xi+\eta), \xi-\eta \ran -\alpha\lan\Gamma(\xi+\eta), \xi-\eta \ran\\
&=\lan(I-K)(\xi+\eta), \xi-\eta \ran+2\alpha\lan\Gamma\xi,\eta\ran,
\endaligned
\de
and
$$
\lan(I-P_\alpha)(\xi+(-\eta)), \xi-(-\eta) \ran=\lan(I-K)(\xi+\eta), \xi-\eta \ran-2\alpha\lan\Gamma\xi,\eta\ran.
$$
Thus,
$$
\lan(I-P_\alpha)(\xi+\eta), \xi-\eta \ran< \lan(I-P_\alpha)(\xi+(-\eta)), \xi-(-\eta) \ran
$$
for $\alpha\in [-1,0]$ and $\eta\in L^2(\pi)$ with $\pi(f\eta)=0$ and $\lan\Gamma \xi,\eta\ran>0$. This means that the supremum in \eqref{vf-g} can not be attained by those $\eta$ such that $\lan\Gamma \xi,\eta\ran>0$, so it is sufficient to consider the function $\eta\in \mathcal{M}_{f,0}$ with $\lan\Gamma \xi,\eta\ran\leq0$. Applying \eqref{expa} to those $\eta$ gives that
$$
\lan(I-P_{\alpha_1})(\xi+\eta),\xi-\eta\ran\geq \lan(I-P_{\alpha_2})(\xi+\eta),\xi-\eta\ran,
$$
for $-1\leq \alpha_1\leq \alpha_2\leq0$. That is, for any $\xi\in L^2(\pi)$,
$$
\sup_{\eta\in \mathcal{M}_{f,0}}\lan(I-P_{\alpha_1})(\xi+\eta),\xi-\eta\ran\geq
\sup_{\eta\in \mathcal{M}_{f,0}}\lan(I-P_{\alpha_2})(\xi+\eta),\xi-\eta\ran.
$$
Hence the second assertion holds by above analysis and \eqref{vf-g}.

\end{proof}


\subsection{A generalization of Peskun's theorem}\lb{appli2}

To prove Theorem \ref{pesk}, we give a little more general result.

Let $P_1$ and $P_2$ be two probability transition kernels on $(E,\e)$ sharing same stationary distribution $\pi$. We say that $P_1$ is smaller than $P_2$, denoted by $P_1\leq P_2$, if
$$
\lan\xi,(I-P_1)\xi\ran\leq \lan\xi,(I-P_2)\xi\ran,\quad \text{for all }f\in L^2(\pi).
$$
Under this ordering and assuming the reversibility of both $P_1$ and $P_2$, \cite{CPS90,Ti98} show that Markov chain $P_2$ is uniformly better than $P_1$, in the sense of having a smaller asymptotic variance. See also \cite[Theorem 2.7]{AL19}.  Now from above variational formulas, we extend it to the Markov chains, where one of them ($P_2$) is allowed to be non-reversible.

\begin{thm}\lb{comp_ord}
Let $P_1$ and $P_2$ be probability transition kernels on $(E,\e)$ sharing same stationary distribution $\pi$. Assume that $P_1$ is reversible, $\lambda(P_2)<1$ and $P_1 \leq P_2$. Then
$$
\widetilde{\sgm}^2(P_2,f)\leq \widetilde{\sgm}^2(P_1,f)\quad \text{for all }f\in L^2_0(\pi).
$$
\end{thm}

\bg{proof}
It suffices to consider the function $f\in L^2_0(\pi)$ such that $\sgm^2(P_1,f)<\infty$. From Theorem \ref{main-vf} and \ref{main-rev} we get
$$
\aligned
1/\sgm^2(P_2,f)
&=\inf_{\xi\in \mathcal{M}_{f,1}}\sup_{\eta\in \mathcal{M}_{f,0}} \lan(I-P_2)( \xi+\eta),\xi-\eta\ran\\
&\geq \inf_{\xi\in \mathcal{M}_{f,1}}\lan(I-P_2)\xi,\xi\ran\\
&\geq \inf_{\xi\in \mathcal{M}_{f,1}}\lan(I-P_1)\xi,\xi\ran\\
&=1/\sgm^2(P_1,f).
\endaligned
$$
Thus combining the above inequality and the definition of $\widetilde{\sgm}^2(\cdot,f)$ gives the desired result.
\end{proof}

Next, we are going to prove Theorem \ref{pesk} as follows.
Recall that $P_1\preceq P_2$ if
$$
P_1(x,A\setminus\{x\})\leq P_2(x,A\setminus\{x\})
$$
for all $A\in \e$ and $\pi$-a.e. $x\in E$.

{\bf Proof of Theorem \ref{pesk}.}
We only need to consider the function $f\in L^2_0(\pi)$ with $\sgm^2(P_1,f)<\infty$. For $P_i,\ i=1,2$ we have
$$
\lan \xi, (I-P_i)\xi\ran=\fr12 \int_E\int_E\pi(dx)P_i(x,dy)\big(f(x)-f(y)\big)^2.
$$
Thus  $P_1 \preceq P_2$ implies that $P_1\le P_2$. So the assertion follows  from Theorem \ref{comp_ord}.
\quad $\square$


As we mentioned in Remark \ref{pesk-non}, it is natural to ask if the Peskun's theorem holds when $P_1$ is non-reversible and $P_2$ is reversible, or both $P_1$ and $P_2$ are non-reversible? Unfortunately, the answer is not, even for finite Markov chains. Here, we give two counter examples.

\bg{exm}\lb{chen}
Consider state space $E=\{1,2,...,6\}$. $P_1$ and $P_2$ are two probability transition matrices on $E$, which are defined by
$$
\centering {\begin{matrix}
P_1=\begin{pmatrix}
     1/2 & 1/2 & 0   & 0   & 0   & 0   \\
     0   & 1/2 & 1/2 & 0   & 0   & 0   \\
     0   & 0   & 1/2 & 1/2 & 0   & 0   \\
     0   & 0   & 0   & 1/2 & 1/2 & 0   \\
     0   & 0   & 0   & 0   & 1/2 & 1/2 \\
     1/2 & 0   & 0   & 0   & 0   & 1/2
\end{pmatrix}, &   P_2 = \begin{pmatrix}
     0   & 1/2 & 0   & 0   & 0   & 1/2   \\
     1/2 & 0   & 1/2 & 0   & 0   & 0   \\
     0   & 1/2 & 0   & 1/2 & 0   & 0   \\
     0   & 0   & 1/2 & 0   & 1/2 & 0   \\
     0   & 0   & 0   & 1/2 & 0   & 1/2 \\
     1/2 & 0   & 0   & 0   & 1/2 & 0
\end{pmatrix}.
\end{matrix}}
$$
\end{exm}

It can be checked that $P_1$ and $P_2$ have same stationary distribution $\pi=(1/6,...,\ 1/6)$, $P_1 \preceq P_2$, and $P_1$ is non-reversible while $P_2$ is reversible.

Now take $f_1=(0,0,1,0,0,-1)^T$. By some calculations, the solutions of Poisson's equation \eqref{poi-int} for chain $P_i(i=1,2)$ and $f_1$ are
$$
\phi_{11}=(3/2,3/2,3/2,-1,-1,-1)^T\ \text{and}\ \phi_{12}=(-1/2,1/2,3/2,1/2,-1/2,-3/2)^T.
$$
Therefore,
$$
\sgm^2(P_1,f_1)=\lan\phi_{11},f_1\ran=5/12<1/2=\lan\phi_{12},f_1\ran=\sgm^2(P_2,f_1).
$$
On the other hand, let $f_2=(1,-1,0,0,0,0)^T$. The solutions of Poisson's equation \eqref{poi-int} for chain $P_i(i=1,2)$ and $f_2$ are
$$
\phi_{21}=(1/3,-5/3,1/3,1/3,1/3,1/3)^T,\ \text{and}\ \phi_{22}=(5/6,-5/6,-1/2,-1/6,1/6,1/2)^T.
$$
Thus we have
$$
\sgm^2(P_1,f_2)=1/3>5/18=\sgm^2(P_2,f_2).
$$
From the above analysis, we see that Peskun's theorem does not hold for $P_1$ and $P_2$.

\bg{exm}\label{two-nonrev}
Consider state space $E=\{1,2,3\}$. $P_1$ and $P_2$ are two probability transition matrices on $E$, which are defined by
$$
\centering {\begin{matrix}
P_1=\begin{pmatrix}
     1/3 & 1/3 & 1/3  \\
     1/4 & 1/2 & 1/4  \\
     0   &  1  & 0
\end{pmatrix}, &   P_2 = \begin{pmatrix}
     0   & 2/3 & 1/3 \\
     3/8 & 3/8 & 1/4 \\
     0   &  1  & 0
\end{pmatrix}.
\end{matrix}}
$$
\end{exm}

Both $P_1$ and $P_2$ cannot be reversible, and it is easy to check that $P_1$ and $P_2$ have same stationary distribution $\pi=(3/14,\ 4/7,\ 3/14)$ and $P_1 \preceq P_2$.
Let $f=(f_1,f_2,f_3)\in L^2_0(\pi)$ with $\pi(f)=0$. Then $f_3=-f_1-\frac{8}{3}f_2$. A
little  calculations shows
$$
\sgm^2(P_1,f)=\frac{1}{294}(126f^2_1+252f_1f_2+448f^2_2);
$$
and
$$
\sgm^2(P_2,f)=\frac{1}{294}(105f^2_1+280f_1f_2+448f^2_2).
$$
Taking $f=(1,1,-11/3)^T$ gives
$$
\sgm^2(P_2,f)-\sgm^2(P_1,f)=\frac{1}{42}>0,
$$
while taking $f=(2,1,-14/3)^T$ gives
$$
\sgm^2(P_2,f)-\sgm^2(P_1,f)=-\frac{2}{21}<0.
 $$


\subsection{Accelerating reversible Markov chains}\lb{appli3}

Peskun's theorem indicates that the reversible transition kernel would be accelerated by some perturbations, which is different from the anti-symmetric perturbations in Sect.\ref{appli1}. We introduce the detail in following.

For convenience, we just consider the finite state space $E=\{1,2,...,N\}$ in this section. Denote $K=\{K_{ij}:i,j\in E\}$ by an irreducible probability transition matrix with stationary distribution $\pi$, and assume it is reversible. We introduce a $N\times N$ matrix $\Lmd$, which has following properties:
\bg{itemize}
\item[($a'$)] $\Lmd\mathbf{1}=0$ and $\mathbf{1}^T\Lmd=0$;

\item[($b'$)] $\Lmd_{ij}\geq 0$, $i\neq j$;

\item[($c'$)] $\pi_iK_{ii}\geq \Lmd_{ii}$, $i\in E$.
\end{itemize}

Using matrix $\Lmd$, we add a perturbation to chain $K$ and obtain a new probability transition matrix as follows.

\bg{lem}\lb{kp}
Suppose that $K$ is a reversible probability transition matrix with stationary distribution $\pi$, and $\Lmd$ is a $N\times N$ matrix , having the properties $(a')-(c')$. Then
\be\lb{non2}
P':=K+\diag(\pi)^{-1}\Lmd
\de
is a probability transition matrix with stationary distribution $\pi$, where $\diag(\pi)$ is defined by the diagonal matrix for vector $\pi$.
\end{lem}

\bg{proof}
From properties $(b)$ and $(c)$, it is obvious that $P'$ is an irreducible probability transition matrix. Now we only need prove that $\pi P'=\pi$. Indeed, together with property $(a')$ and $\pi K=\pi$ gives
$$
\pi P'=\pi K+ \mathbf{1}^T \Lmd =\pi.
$$
\end{proof}

By the definitions of $\Lmd$ and $P'$, it is easy to see that $K\preceq P'$ in Peskun's order. So we can compare the asymptotic variance of chains $K$ and $P'$ from Theorem \ref{pesk}.

\bg{cor}\lb{another}
Suppose that $K$ is a reversible probability transition kernel with stationary distribution $\pi$, and $P'$ is defined in \eqref{non2}. Then
$$
\widetilde{\sgm}^2(P',f)\leq \widetilde{\sgm}^2(K,f),\q \text{for all}\  f\in L_0^2(\pi).
$$
\end{cor}

\bg{rem}
It is obvious that chain $P'$ is reversible if $\Lmd$ is symmetric. In this special case, a result analogous to Corollary \ref{another} is established in \cite[Theorem 3]{CH13}.
\end{rem}


In fact, we also can prove that chain $P'$ mixes faster than chain $K$ in another terms under some conditions. For that, we introduce some notations. Let $\mathcal{M}$ be the class of nonnegative non-increasing functions on $E$. We introduce another partial ordering for probability transition matrices, which was introduced in Fill and Kahn \cite{FK13}.
Let $P$ and $Q$ be two probability transition matrices, with same stationary distribution $\pi$,  we declare that $P \ll Q$ if
$$
\langle P \xi,\eta \rangle \leq \langle Q\xi,\eta \rangle,\quad \text{for every } \xi,\ \eta\in\mathcal{M}.
$$

In following proposition, we obtain an interesting fact that Peskun ordering can derive the above partial ordering.

\bg{prop}\lb{pefi}
Let $P$ and $Q$ be two probability transition matrices sharing same stationary distribution $\pi$. If $P \preceq Q$, then $Q\ll P$.
\end{prop}

To prove Proposition \ref{pefi}, we need the following lemma.

\bg{lem}\lb{fillequa}
Let $P$ and $Q$ be two probability transition matrices sharing same stationary distribution $\pi$. Then $P \ll Q$ if and only if
$$
\sum_{i=1}^n \sum_{j=1}^m \pi_i P_{ij}\leq \sum_{i=1}^n \sum_{j=1}^m \pi_i Q_{ij},\quad {for}\ n, m=1,\cdots, N.
$$
\end{lem}
\bg{proof}
From Fill and Kahn \cite[Remark 2.1]{FK13} it follows that $P\ll Q$ if and only if
$
\langle P\xi,\eta \rangle \leq \langle Q\xi,\eta \rangle,
$
where $\xi,\ \eta$ are carried over all indicator functions of down-sets(i.e., sets $D$ such that $j\in D$ and $i\leq j$ implies $i\in D$). Now if we take $\xi=\mathbf{1}_{[1,m]}$ and $\eta=\mathbf{1}_{[1,n]},\ m,n\leq N$, then
$$
\langle P\xi,\eta\rangle=\sum_{i=1}^n \sum_{j=1}^m \pi_i P_{ij}.
$$
So the above analysis gives the desired result.
\end{proof}

Now we are ready to prove Proposition \ref{pefi}.

\medskip
\noindent {\bf Proof of Proposition \ref{pefi}.}
Define $R=Q-P$. Since $P \preceq Q$, $R$ has following properties: (i) $R_{ij}\geq 0$, for $i\neq j$; (ii) $R_{ii}\leq 0$, for all $i\in E$; (iii) $R\mathbf{1}=0$ and $\pi R=0$. For convenience, denote $\Lambda=\text{diag}(\pi)R$. We can see that $\Lambda$ has properties (i), (ii) and (iii') $\Lambda\mathbf{1}=0$ and $\mathbf{1}^T\Lambda=0$. Thus for any $m,n\leq N$,
\be\lb{qpine}
\aligned
\sum_{i=1}^n \sum_{j=1}^m \pi_i Q_{ij}&=\sum_{i=1}^n \sum_{j=1}^m (\pi_i P_{ij}+\Lambda_{ij})\\
&\leq \sum_{i=1}^n \sum_{j=1}^m \pi_i P_{ij},
\endaligned
\de
since $\sum_{i=1}^n \sum_{j=1}^m \Lambda_{ij}\leq 0$ by the properties (i), (ii), (iii'). Combining \eqref{qpine} and Lemma \ref{fillequa} gives the desired result.
\quad $\square$

We say that a probability transition matrix $P$ is stochastically monotone if $Pf\in \mathcal{M}$ for every $f\in \mathcal{M}$. We also need the concept of majorization, see Marshall, Olkin and Arnold \cite{MOA79} for more backgrounds. Given two vectors $v$ and $w$ in $\mathbb{R}^N$, if (i) for each $i=1,\cdots, N$ the sum of the $i$ largest entries of $w$ is at least the corresponding sum for $v$, and (ii) equality holds when $i=N$, then we say that $v$ majorizes $w$.

\bg{cor}\lb{major}
Let $\Lambda$ be a symmetric matrix which has properties $(a')-(c')$, and $P'$ be defined in \eqref{non2}. Suppose that $K,\ P'$ are stochastically monotone, and their stationary distribution $\pi$ is non-increasing. If $X$ and $Y$ are chains (1) started at a common probability $\hat{\pi}$ such that $\hat{\pi}/\pi$ is non-increasing and (2) having respective transition matrices $K$ and $P'$, then for all $n$, the probability measure function $\rho_n$ of $X_n$ majorizes the probability measure function $\varrho_n$ of $Y_n$.
\end{cor}

\bg{proof}
The desired conclusion follows immediately upon combining Proposition \ref{pefi} and \cite[Corollary 3.7]{FK13}.
\end{proof}

\bg{rem}
\begin{itemize}
\item[(1)] In fact, for two transition matrix $P$ and $Q$ with same stationary distribution $\pi$ and $P\preceq Q$, there must exist a matrix $\Lambda$, satisfying $(a')-(c')$, such that
    $$
    Q=P+\text{diag}(\pi)^{-1}\Lambda.
    $$
    So the result in Corollary \ref{major} holds for chains $P$ and $Q$.

\item[(2)] From the properties of majorization and Corollary \ref{major}, we can see that chain $Y$ mixes faster than does $X$ in senses of $L^p-$distance($1\leq p\leq \infty$), separation, Hellinger distance, etc., see \cite[Chapter 3]{MOA79} and \cite[Example 3.8]{FK13} for more details.
\end{itemize}
\end{rem}

Unfortunately, although the ordering ``$\ll $'' allows some efficiency comparisons among different transition matrices, if we replace Peskun ordering by this ordering in Theorem \ref{pesk}, the comparison result does not hold. We give a counterexample for that.

\bg{exm}
Consider state space $E=\{1,2,3\}$. $P$ and $Q$ are probability transition matrices on $E$, which are defined by
$$
\centering {\begin{matrix}
P=\begin{pmatrix}
      0  & 1/2 & 1/2  \\
     1/2 &  0  & 1/2  \\
     1/2 & 1/2 & 0
\end{pmatrix}, &   Q = \begin{pmatrix}
     0   & 1/3 & 2/3 \\
     1/3 & 1/3 & 1/3 \\
     2/3 & 1/3  & 0
\end{pmatrix}.
\end{matrix}}
$$
By straightforward calculations, chains $P$ and $Q$ have same stationary distribution $\pi=(1/3,1/3,1/3)$ and $P\ll Q$. Fix function $f=(f_1,f_2,f_3)\in L^2_0(\pi)$, we have
$$
\sgm^2(P,f)=\frac{4}{9}(f_1 ^2+f_1 f_2+f_2 ^2);
$$
$$
\sgm^2(Q,f)=\frac{2}{5}f_1 ^2+\frac{3}{5}f_1 f_2+\frac{2}{5}f_2 ^2.
$$
It is easy to see that the uniformly comparison result in Theorem \ref{pesk} of the asymptotic variance does not hold between chains $P$ and $Q$.
\end{exm}

\section{Conclusion}\lb{conclu}

In this paper, we establish some variational formulas for asymptotic variance of Markov chains. As applications, we extend Peskun's theorem to non-reversible case, and obtain some comparison theorems of non-reversible and reversible Markov chains in terms of asymptotic variance. According to new Peskun's theorem, we provide a new method to reduce the asymptotic variance of reversible Markov chains.

Now we have two ways to accelerate reversible Markov chains in the term of asymptotic vairance, which are introduced in Sect.\ref{appli1} and \ref{appli3}, respectively. From the definitions of $\Gm$ and $\Lambda$, we can see that anti-symmetric perturbation $\Gm$ reduces the times that chain makes the round on cycles, while the drift $\Lambda$ reduces the chances that chain stays at one state for a long time. We are interested in comparing this two methods. But as we know, it is complicated. We only give some examples in following to illustrate that.

\bg{exm}
Suppose that the state space $E=\{1,2,3,4\}$. Consider a reversible Markov chain with probability transition matrix
$$
K=\begin{pmatrix}
     0 & 1/2 & 0 & 1/2  \\
     1/2 & 0 & 1/2 & 0  \\
     0   &  1/2  & 0 & 1/2 \\
     1/2 & 0 & 1/2 & 0
\end{pmatrix}.
$$
This is an example in \cite{CH13}. It is easy to work out that its stationary distribution $\pi=(1/4,1/4,1/4,1/4)$. Now we add a vorticity matrix to $K$ and obtain
$$
P=K+ \diag(\pi)^{-1}\begin{pmatrix}
     0 & 1/8 & 0 & -1/8  \\
     -1/8 & 0 & 1/8 & 0  \\
     0   &  -1/8  & 0 & 1/8 \\
     1/8 & 0 & -1/8 & 0
\end{pmatrix}
=\begin{pmatrix}
     0 & 1 & 0 & 0  \\
     0 & 0 & 1 & 0  \\
     0   &  0  & 0 & 1 \\
     1 & 0 & 0 & 0
\end{pmatrix},
$$
which is uniformly better than chain $K$ in the term of asymptotic variance. But since $K$ has not more than one nonzero diagonal entry, the method in Sect.\ref{appli3} does not work.
\end{exm}

There also exists some reversible chains which has not cycle on its graph, so the method in Sect.\ref{appli1} does not work.

\bg{exm}
Suppose that the state space $E=\{1,2,3\}$. Consider a probability transition matrix
$$
K=\begin{pmatrix}
     2/3 & 1/3 & 0  \\
     1/3 & 1/3 & 1/3  \\
       0 & 1/3 & 2/3  \\
\end{pmatrix}.
$$
It is easy to work out that its stationary distribution $\pi=(1/3,1/3,1/3)$ and it is reversible. Furthermore, we can see that there does not exist anti-symmetric perturbations to accelerate it. But from Corollary \ref{another}, we know that
$$
P=K+ \diag(\pi)^{-1}\begin{pmatrix}
     -1/9 & 1/9  & 0  \\
      1/9 & -1/9 & 0  \\
     0    & 0    & 0  \\
\end{pmatrix}
=\begin{pmatrix}
     1/3 & 2/3 & 0  \\
     2/3 & 0 & 1/3  \\
       0 & 1/3 & 2/3  \\
\end{pmatrix}
$$
is uniformly better than chain $K$ in terms of the asymptotic variance.
\end{exm}

Maybe the above two examples are little trivial, we give another example which makes the two methods vaild in following.

\bg{exm}
Suppose that the state space $E=\{1,2,3\}$. Consider a reversible Markov chain with transition matrix
$$
K=\begin{pmatrix}
     1/3 & 1/3 & 1/3  \\
     1/3 & 1/3 & 1/3  \\
     1/3 & 1/3 & 1/3  \\
\end{pmatrix}.
$$
Its stationary distribution $\pi=(1/3,1/3,1/3)$. Let vorticity matrix $\Gm$ be
$$
\Gm=\begin{pmatrix}
     0 & -1/9 & 1/9  \\
     1/9 & 0 & -1/9  \\
     -1/9 & 1/9 & 0  \\
\end{pmatrix}.
$$
And we can obtain a probability transition matrix $P$ from Lemma \ref{ptk}. It is easy to check that $\Gm$ is the best among all vorticity matrices satisfying properties $(a)-(c)$ and
$$
\sgm^2(P,f)=\frac{1}{2}(f^2_1+f_1f_2+f^2_2)
$$
for $f=(f_1,f_2,f_3)\in L^2_0(\pi)$.

Next, we want to accelerate chain $K$ by the second method. We introduce matrices $\Lmd_1$ and $\Lmd_2$ as

$$
\centering{\begin{matrix}
\Lmd_1=\begin{pmatrix}
     -1/9  & 1/9  & 0 \\
      1/9  &-1/9  & 0 \\
        0  &   0  & 0  \\
\end{pmatrix}, &
\Lmd_2=\begin{pmatrix}
     -1/9 & 1/9 & 0  \\
     0 & -1/9 & 1/9  \\
     1/9 & 0 & -1/9  \\
\end{pmatrix}.
 \end{matrix}}
$$
From \eqref{non2}, we can construct two probability transition matrices $P'_1$ and $P'_2$.  By some calculations,
$$
\sgm^2(P^{'}_1,f)=\frac{3}{5}f^2_1+\frac{4}{5}f_1f_2+\frac{3}{5}f^2_2;
$$
$$
\sgm^2(P^{'}_2,f)=\frac{3}{7}(f^2_1+f_1f_2+f^2_2).
$$

So it is obvious that $P'_2$ is uniformly better than $P$ in the terms of asymptotic variance. but there is not uniformly comparison result for chains $P'_1$ and $P$, even for chains $P'_1$ and $P'_2$, it is surprising.
\end{exm}

{\bf Acknowledgement}\
Lu-Jing Huang would like to thank Professor Chii-Ruey Hwang, Ting-Li Chen and Dr. Michael C.H. Choi for helpful discussions (Example \ref{chen} is from Professor Ting-Li Chen). Lu-Jing Huang acknowledges support from NSFC (No. 11901096), NSF-Fujian(No. 2020J05036), the Program for Probability and Statistics: Theory and Application (No. IRTL1704), and the Program for Innovative Research Team in Science and Technology in Fujian Province University (IRTSTFJ). Yong-Hua Mao acknowledges support from NSFC (No. 11771047) and National Key Research and Development Program of China (2020YFA0712901).

\bibliographystyle{plain}
\bibliography{VF}

\end{document}

So $P$ admits a spectral decomposition:
$$
P=\int_{-1}^1 xdE_x.
$$
For $f\in L^2(\pi)$, denote measure $\nu_f(dx)=d\lan E_xf,f\ran$. The following FCLT for reversible Markov chain is from \cite{KV86}(see also \cite[Theorem 1.10]{KLO12}).

\begin{prop}
Let $X$ be an irreducible reversible and ergodic Markov chain on $E$, with probability transition kernel $P$ and stationary distribution $\pi$. Let $f$ be a function in $L^2_0(\pi)$ with $\int_{-1}^1 1/(1-x)\mu_f(dx)<\infty$. Then under $\mathbb{P}_\pi$,
$$
\frac{1}{\sqrt{n}}\sum_{k=0}^{n-1}f(X_k)\xrightarrow{d} N\big(0,\tilde{\sgm}^2(P,f)\big),
$$
where $\widetilde{\sgm}^2(P,f)$ is defined by
$$
\widetilde{\sigma}^2(P,f)=\int_{-1}^1\frac{1+x}{1-x}\mu_f(dx)=2\int_{-1}^1 \frac{1}{1-x}\mu_f(dx)-||f||^2.
$$
\end{prop}

From above Proposition, we see that for $f\in L^2_0(\pi)$, $\sigma^2(P,f)$ has following spectral representation:

$$
\sigma^2(P,f)=\int_1^{-1}\frac{1}{1-x}\mu_f(dx).
$$